\documentclass[11 pt, reqno]{amsart}
\usepackage{fullpage}
\usepackage{amsthm,amssymb,amsfonts,hyperref, mathrsfs}
\hypersetup{
    colorlinks=true,       % false: boxed links; true: colored links
    linkcolor=blue,          % color of internal links (change box color with linkbordercolor)
    citecolor=cyan,        % color of links to bibliography
}
\theoremstyle{plain}
\def\ssum{\mathop{\sum\!\sum}}

\newcommand{\sumstar}{\sideset{}{^{*}}\sum}

\def\rb{\mathbb{R}}
\def\nb{\mathbb{N}}
\def\zb{\mathbb{Z}}
\def\cb{{\mathbb C}}
\renewcommand{\mod}{\mathop{\rm{mod}}}
\numberwithin{equation}{section}
\newtheorem{theorem}{Theorem}[section]
\newtheorem{lemma}[theorem]{Lemma}
\newtheorem{proposition}[theorem]{Proposition}

\begin{document}

\title{The sum of multidimensional divisor function over values of quadratic polynomial}

\author{Nianhong Zhou}
\address{Department of Mathematics\\
East China Normal University\\
500 Dongchuan Road, Shanghai 200241, PR China}
\email{nianhongzhou@outlook.com}
\begin{abstract}
Let $F({\bf x})={\bf x}^tQ_m{\bf x}+\mathbf{b}^t{\bf x}+c\in\mathbb{Z}[{\bf x}]$ be a quadratic polynomial in $\ell (\ge 3 )$ variables ${\bf x} =(x_{1},...,x_{\ell})$, where $F({\bf x})$ is positive when ${\bf x}\in\mathbb{R}_{\ge 1}^{\ell}$, $Q_m\in {\rm M}_{\ell}(\mathbb{Z})$ is an $\ell\times\ell$ matrix  and its discriminant $\det\left(Q_m^t+Q_m\right)\neq 0$.
It gives an asymptotic formula for the following sum
\[
T_{k,F}(X)=\sum_{{\bf x}\in [1,X]^{\ell}\cap\mathbb{Z}^{\ell}}\tau_{k}\left(F({\bf x})\right)
\]
with the help of the circle method. Here $\tau_{k}(n)=\#\{(x_1,x_2,...,x_{k})\in\nb^{k}: n=x_1x_2...x_{k}\}$ with $k\in\zb_{\ge 2}$ is the multidimensional divisor function.
\end{abstract}
\maketitle

%\linenumbers

\section{Introduction}
The multidimensional divisor functions are generalisations of the divisor function $\tau(n)=\sum_{d|n}1$, defined by
\[
\tau_{k}(n)=\#\{(x_1,x_2,...,x_{k})\in\nb^{k}: n=x_1x_2...x_{k}\},
\]
and counting the number of ways that $n$ can be written as a product of $k$ positive integer numbers. Understanding the average order of $\tau_k(n)$, as it ranges over the values taken by polynomials is an important topic in analytic number theory.
The behavior of $\tau_k(n)$ is far less than perfectly understood even for $k=3$. For example, so far there are no
asymptotic formulas for the sum $\sum_{m\le x}\tau_{3}(m^{2}+1)$.  We considers the sum
\[\sum_{\left|F({\bf x})\right|\le x}\tau_k(\left|F({\bf x})\right|),\]
where $F({\bf x})\in \zb[x_1,x_2]$ is a binary form. For $k=2$ and $F({\bf x})$ is an irreducible cubic form, Greaves \cite{MR0263761} showed that there exists constants $c_0,c_1\in\rb$ with $c_0>0$ depending only on $F$, such that
\[
\sum_{\left|F({\bf x})\right|\le X}\tau(\left|F({\bf x})\right|)= c_0X^{\frac{2}{3}}\log X+c_1X^{\frac{2}{3}} +O_{\varepsilon, F}(X^{\frac{9}{14}+\varepsilon}),
\]
holds for any $\varepsilon>0$. If $F$ is an irreducible quartic form, Daniel \cite{MR1670278} showed that
\[
\sum_{\left|F({\bf x})\right|\le X}\tau(\left|F({\bf x})\right|)= c_2X^{\frac{1}{2}}\log X+O_{F}(X^{\frac{1}{2}}\log\log X),
\]
where $c_2$ is a constant depending only on $F$. For more related works, see e.g. la Bret{\`{e}}che, Browning \cite{MR2719554} and  Browning\cite{MR2861076}. However, if $k\ge 3$, this kind of problems will become more complicated. There are few results in this direction. For $\tau_3(n)$, Friedlander and Iwaniec \cite{MR2289206} showed that
\[
\sumstar_{n_1^2+n_2^6\leq x}\tau_3(n_1^2+n_2^6)=cx^{\frac{2}{3}}(\log x)^2
+O\left(x^{\frac{2}{3}}(\log x)^{\frac{7}{4}}(\log \log x)^{\frac{1}{2}}\right),
\]
where $c$ is a constant and $*$ means that $(n_1,n_2)=1$. If $F({\bf x})$ is a positive definite quadratic form with $\ell\ge 2$ variables, then it is easy to obtain the sum
\[\sum_{F({\bf x})\le X}\tau_k(F({\bf x}))=\sum_{m\le x}\tau_k(m)r_F(m)\]
by classical results of quadratic form, where $r_F(m)=\#\{{\bf x}\in\zb^{\ell}: m=F({\bf x})\}$. For example, Yu \cite{MR1754029} obtained the following
\[
\sum_{1\le m_1, m_2\le X}\tau\left(m_1^2+m_2^2\right)=c_4X^2\log X+c_5X^2+O_{\varepsilon}(X^{\frac{3}{2}+\varepsilon}),
\]
Sun and Zhang \cite{MR3515816} obtained the following
\[
\sum_{1\le m_1, m_2, m_3\le X}\tau_3\left(m_1^2+m_2^2+m_3^2\right)=c_6X^3\log^2 X+c_7X^3\log X+c_8X^3+O_{\varepsilon}(X^{\frac{11}{4}+\varepsilon}),
\]
where $c_4, c_5, c_6,c_7,c_8$ are constants and $\varepsilon$ is an arbitrarily positive number. However this method does not working if $F$ is an indefinite quadratic form. On the other hand, nothing of the following sum
\begin{equation}\label{T}
T_{k,F}(X):=\sum_{{\bf x}\in [1,X]^{\ell}\cap \zb^{\ell}}\tau_k(F({\bf x})),
\end{equation}
is known for $k\ge 4$ and quadratic form $F$ is positive definite or indefinite.\newline

The purpose of this paper is to investigate general problem as above. More precisely, let $F({\bf x})$ be a quadratic polynomial with $\ell(\ge3)$ variables $x_1, x_2,.., x_{\ell}$ and integer coefficients. The vector
${\bf x}=(x_1, x_2, . . . , x_{\ell})^{t}\in\mathbb{Z}^{\ell}$ and denote $B_{\ell}(X)=[1, X]^{\ell}\cap \zb^{\ell}$ as a box for some sufficiently large positive number $X$. Also assume quadratic polynomial $F({\bf x})$ satisfies
\begin{equation}\label{fdef1}
F({\bf x})={\bf x}^tQ_m{\bf x}+\mathbf{b}^t{\bf x}+c,
\end{equation}
where $Q_m\in {\rm M}_{\ell}(\mathbb{Z})$ is an $\ell\times\ell$ matrix with entries $a_{ij}$, vector $\mathbf{b}=(b_1,..., b_{\ell})^t\in\mathbb{Z}^{\ell}$, $c\in\zb $ and suppose those coefficients satisfy the following
\begin{align}\label{fdef2}
\begin{cases}
\min_{{\bf x}\in B_{\ell}(X)}F({\bf x})>0\\
\Delta_{F}=\det\left(Q_{m}^t+Q_m\right)\neq 0.
\end{cases}
\end{align}
Thus $F({\bf x})$ has a maximum value $N_F(X)$ in the box $B_{\ell}(X)$ when $X$ is sufficiently large, say
\begin{equation}\label{qmax1}
N_{F}(X)=X^2\sum_{1\le i, j\le \ell}a_{ij}+X\sum_{1\le r\le \ell}b_{r}+c.
\end{equation}
Our main result is the following.
\begin{theorem} Let $F$, $B(X)$ be defined as above, $k\ge 2$ and $\ell\ge 3$. For any $\varepsilon>0$ there exist constants $H_{k,0}(F)$, $H_{k,1}(F)$,..., and $H_{k,k-1}(F)$, such that
\[T_{k,F}(X)=\sum_{r=0}^{k-1}H_{k,r}(F)\int_{[1,X]^{\ell}}(\log F({\bf t}))^r{\rm d}{\bf t}+O_{\varepsilon,k,F}\left(X^{\ell-\frac{\ell-2}{\ell+2}\min\left(1,\frac{4}{k+1}\right)+\varepsilon}\right)\]
and
\[H_{k,r}(F)=\frac{1}{r!}\sum_{t=0}^{k-r-1}\frac{1}{t!}\left({\frac{{\rm d}^tL(s;k,F)}{{\rm d}s^t}}\bigg|_{s=1}\right) {\rm Res}\left((s-1)^{r+t}\zeta(s)^k; s=1\right),\]
where the function $L(s; k,F)$ is given in Lemma \ref{t41}.
\end{theorem}
\subsection*{Notation}
The symbols  $\mathbb{N}$, $\mathbb{Z}$ and $\mathbb{R}$ denote the positive integers, the integers and the real numbers, respectively. $e(z)=e^{2\pi i z}$, the letter $p$ always denotes a prime, $M^t$ is transpose operation of matrix $M$. The symbol $\zb_q$
represents shorthand for the groups $\zb/q\zb$. Also, the shorthand for the multiplicative
group reduced residue classes $(\zb/q\zb)^*$ is $\zb_q^*$. Occasionally we make use of the $\varepsilon$-convention: whenever $\varepsilon$ appears in a statement, it is asserted that the statement is true for all real $\varepsilon $. This allows us to write $x^{\varepsilon}\log x\ll x^{\varepsilon}$ , for example.
\section{Primaries}
The primary technique used in the proof of the main theorem is the circle method. We shall need the following results which need in the sequel. Lemma \ref{t24}, Lemma \ref{t25} and Lemma \ref{t26} will be used in the estimates of the major arcs of the circle method. Lemma \ref{t28} will be used in the estimate of the minor arcs. To obtain Lemma \ref{t24}, we firstly need the follows lemma.
\begin{lemma} [R. A. Smith] \label{t21}Let $1\le h\le q$, $(q,h)=\delta$. Then for $q\le x^{\frac{2}{k+1}}$ , we have
\[
\sum_{\substack{m\le x\\ m\equiv h~(\mod q)}}\tau_{k}(m)=M_{k}(x;h,q)+O_k(x^{1-\frac{2}{k+1}+\varepsilon}),
\]
where
\[
M_{k}(x;h,q)={\rm Res}\left(\zeta(s)^{k}\frac{x^{s}}{s}f_{k}(q, \delta,s); s=1\right)
\]
with
\begin{equation}\label{fkqdef}
f_{k}(q,\delta,s)=\frac{1}{\varphi(q/\delta)\delta^s}\prod_{p|(q/\delta)}\left(1-\frac{1}{p^s}\right)^k\sum_{\substack{d_1d_2...d_{k}=\delta\\ d_1,d_2,...,d_{k}>0}}\prod_{i=1}^{k-1}\prod_{\substack{p|(\prod_{r=i+1}^{k}d_r)\\ (p,q/\delta)=1}}\left(1-\frac{1}{p^s}\right).
\end{equation}
\end{lemma}
\begin{proof} This lemma is essentially made by Smith \cite{MR664379}, and we just change the form as needed. Firstly, by the equation (30) of \cite{MR664379}, we get
\[
A_k(x;h,q)=\sum_{\substack{d_1d_2...d_{k}=\delta\\ d_1,d_2,...,d_{\ell}>0}}\sum_{\substack{t_i|\prod_{r={i+1}}^{k}d_{r}\\ i=1,2,...,k\\(t_1t_2...t_{k},q/\delta)=1}}\mu({\bf t})A_k\left(\frac{x}{\delta t_1t_2...t_{k}};\overline{t_1t_2...t_kh/\delta},q/\delta\right)\\
\]
where the notations be followed. Theorem 3 of this paper yields
\[A_k(x;h,q)=M_{k}(x;h,q)+\Delta_k(x;h,q),\]
where
\[
M_{k}(x;h,q)=\sum_{\substack{d_1d_2...d_{k}=\delta\\ d_1,d_2,...,d_{k}>0}}\sum_{\substack{t_i|\prod_{r={i+1}}^{k}d_{r}\\ i=1,2,...,k\\(t_1t_2...t_{k},q/\delta)=1}}\mu({\bf t})\frac{x}{\delta t_1t_2...t_{k}}P_{k}\left(\log\left(\frac{x}{\delta t_1t_2...t_{k}}\right),\frac{q}{\delta}\right)
\]
and
\[
\Delta_k(x;h,q)=\sum_{\substack{d_1...d_{k}=\delta\\ d_1,...,d_{k}>0}}\sum_{\substack{t_i|\prod_{r={i+1}}^{k}d_{r}\\ i=1,2,...,k\\(t_1...t_{k},q/\delta)=1}}\mu({\bf t})\left(D_k\left(0;\overline{\left(\frac{t_1...t_kh}{\delta}\right)},\frac{q}{\delta}\right)+O\left(\frac{\tau_{k}\left({q}/{\delta}\right)x^{\frac{k-1}{k+1}}\log^{k}x}{(\delta t_1...t_{k})^{\frac{k-1}{k+1}}\log x}\right)\right).
\]
By the definition of $P_k(\log x, q)$, namely (13),(21) and relatively talking about (21)  of \cite{MR664379}. It is easily seen that
\[
xP_k(\log x, q)=\frac{1}{\varphi(q)}{\rm Res}\left(\left(\zeta(s)\sum_{d|q}d^{-s}\mu(d)\right)^{k}\frac{x^{s}}{s}; s=1\right).
\]
Hence we obtain that
\begin{align*}
M_{k}(x;h,q)&=\sum_{\substack{d_1...d_{k}=\delta\\ d_1,...,d_{k}>0}}\sum_{\substack{t_i|\prod_{r={i+1}}^{k}d_{r}\\ i=1,2,...,k\\(t_1...t_{k},q/\delta)=1}}\frac{\mu({\bf t})}{\varphi(q/\delta)}{\rm Res}\left(\left(\zeta(s)\sum_{d|(q/\delta)}\frac{\mu(d)}{d^s}\right)^{k}\frac{x^{s}}{s}\frac{1}{\left(\delta t_1...t_{k}\right)^{s}}; s=1\right)\\
&={\rm Res}\left(\frac{x^{s}}{s}\frac{\zeta(s)^{k}}{\varphi(q/\delta)\delta^s }\prod_{p|(q/\delta)}\left(1-\frac{1}{p^s}\right)^k\sum_{\substack{d_1d_2...d_{k}=\delta\\ d_1,d_2,...,d_{k}>0}}\sum_{\substack{t_i|\prod_{r={i+1}}^{k}d_{r}\\ i=1,2,...,k\\(t_1t_2...t_{k},q/\delta)=1}}\frac{\mu({\bf t})}{\left( t_1...t_{k}\right)^{s}}; s=1\right)\\
&:={\rm Res}\left(\zeta(s)^{k}\frac{x^{s}}{s}f_{k}(q,\delta,s); s=1\right),
\end{align*}
where
\begin{align*}
f_{k}(q,\delta,s)&=\frac{1}{\varphi(q/\delta)\delta^s}\prod_{p|(q/\delta)}\left(1-\frac{1}{p^s}\right)^k\sum_{\substack{d_1d_2...d_{k}=\delta\\ d_1,d_2,...,d_{k}>0}}\sum_{\substack{t_i|\prod_{r={i+1}}^{\ell}d_{r}\\ i=1,2,...,\ell\\(t_1t_2...t_{k},q/\delta)=1}}\frac{\mu({\bf t})}{\left(t_1...t_{k}\right)^{s}}\\
&=\frac{1}{\varphi(q/\delta)\delta^s}\prod_{p|(q/\delta)}\left(1-\frac{1}{p^s}\right)^k\sum_{\substack{d_1d_2...d_{k}=\delta\\ d_1,d_2,...,d_{k}>0}}\prod_{i=1}^{k-1}\prod_{\substack{p|\prod_{r=i+1}^{k}d_r\\ (p,q/\delta)=1}}\left(1-\frac{1}{p^s}\right).
\end{align*}
Smith \cite{MR664379} conjectured the validity of the estimate $D(0,h,q)\ll q^{\frac{k-1}{2}+\varepsilon}$ for any $(q,h)=1$ and proved by Matsumoto \cite{MR792769}. Which implies the bound
\begin{align*}
\Delta_k(x;h,q)&\ll\sum_{\substack{d_1...d_{k}=\delta\\ d_1,...,d_{k}>0}}\sum_{\substack{t_i|\prod_{r={i+1}}^{k}d_{r}\\ i=1,2,...,k}}\left|\mu({\bf t})\right|\left(\left({q}/{\delta}\right)^{\frac{k-1}{2}+\varepsilon}+q^{\varepsilon}x^{\frac{k-1}{k+1}+\varepsilon}\right)\\
&\ll_k \left(q^{\frac{k-1}{2}+\varepsilon}+x^{\frac{k-1}{k+1}+\varepsilon}\right)\sum_{d_1...d_{k}=\delta}\tau(\delta)^{k-1}\ll_{k} x^{1-\frac{2}{k+1}+\varepsilon}.
\end{align*}
This completes the proof of the lemma.
\end{proof}

We have the proposition which will be used in the proof of Lemma \ref{t24}.
\begin{proposition}\label{t22} Let $q\ge 1$ be an integer, $(a,q)=1$ and denote $\delta=(h,q)$. Also let $f(q,\delta, s)$ be defined as in Lemma \ref{t21}. Define
\[F_{k,a}(q,s)=\sum_{h\in\zb_q}e\left(-\frac{ah}{q}\right)f_k(q,\delta,s).\]
Then $F_{k,a}(q,s)$ is independent of $a$ and we may write it as $F_k(q,s)$.
Furthermore, $F_k(q,s)$ is multiplicative function and
\[\frac{{\rm d}^s F_k(q,1)}{{\rm d}s^r}\ll_{k} q^{-1+\varepsilon}\]
holds for any integer $r=0,1,...,k-1$.
\end{proposition}
\begin{proof}
First, we have
\begin{align*}
F_{k,a}(q,s)&=\sum_{\delta|q}\sum_{\substack{h\in\zb_q\\ (h,q)=\delta}}e\left(-\frac{ah}{q}\right)f_k(q,\delta,s)=\sum_{\delta|q}f_k(q,\delta,s)\sum_{h_1\in\zb_{q/\delta}^*}e\left(-\frac{ah_1}{q/\delta}\right)\\
&=\sum_{\delta|q}c_{\delta}(a)f_k(q,q/\delta,s)=\sum_{\delta|q}\mu(\delta)f_k(q,q/\delta,s),
\end{align*}
where $c_{\delta}(a)$ is the Ramanujan's sum and the fact that if $(a,\delta)=1$ then $c_{\delta}(a)=\mu(\delta)$ be used. This result yields $F_{k,a}(q,s)$ independent on $a$. Suppose that positive integers $q_1$ and $q_2$ are coprime, then
\begin{align*}
F_k(q_1,s)F_k(q_2,s)&=\sum_{\delta_2|q_2}\sum_{\delta_1|q_1}\mu(\delta_1)\mu(\delta_2)f_k(q_1,q_1/\delta_1,s)f_k(q_2,q_2/\delta_2,s)\\
&=\sum_{(\delta_1\delta_2)|(q_1q_2)}\mu(\delta_1\delta_2)f_k(q_1,q_1/\delta_1,s)f_k(q_2,q_2/\delta_2,s),
\end{align*}
hence we just need to show
\[f_k(q_1,q_1/\delta_1,s)f_k(q_2,q_2/\delta_2,s)=f_k(q_1q_2,q_1q_2/(\delta_1\delta_2),s)\]
whenever $\delta_1|q_1$ and $\delta_2|q_2$. It is obtained by the definition of $f_k(q,q/\delta,s)$, say
\[
f_k(q,q/\delta,s)=\frac{\delta^s}{\varphi(\delta)q^s}\prod_{p|\delta}\left(1-\frac{1}{p^s}\right)^k\sum_{\substack{d_1d_2...d_{k}=q/\delta\\ d_1,d_2,...,d_{k}>0}}\prod_{i=1}^{k-1}\prod_{\substack{p|(\prod_{r=i+1}^{k}d_r)\\ (p,\delta)=1}}\left(1-\frac{1}{p^s}\right).\]
Furthermore,
\[
f_k(q,q/\delta,s)\ll \frac{\delta^{\sigma}}{\varphi(\delta)q^{\sigma}}\prod_{p|\delta}\left(1+\frac{1}{p^{\sigma}}\right)^k\sum_{\substack{d_1d_2...d_{k}=q/\delta\\ d_1,d_2,...,d_{k}>0}}\prod_{i=1}^{k-1}\prod_{\substack{p|(\prod_{r=i+1}^{k}d_r)\\ (p,\delta)=1}}\left(1+\frac{1}{p^{\sigma}}\right),
\]
where $\sigma={\rm Re}(s)$. It is easily seen that if $s=1+\rho e(\theta)$ with $\theta\in[0,1)$, then
\[f_k(q,q/\delta,s)\ll \frac{\delta^{\sigma}}{\varphi(\delta)q^{\sigma}} 2^{k\omega(\delta)}\tau_k(q)2^{(k-1)\omega(q)}\ll q^{\varepsilon} \frac{\delta^{\sigma}}{\varphi(\delta)q^{\sigma}}.\]
Thus we have
\[F_{k}(q,s)\ll q^{\varepsilon}\sum_{\delta|q}\left|\mu(\delta)\right|\frac{\delta^{\sigma}}{\varphi(\delta)q^{\sigma}}= q^{-\sigma+\varepsilon}\prod_{p|q}\left(1+\frac{p^{\sigma}}{p-1}\right)\ll q^{-\sigma+\varepsilon}\prod_{p|q}\left(1+\frac{p^{\sigma}}{p}\right).\]
On the other hand
\begin{align*}
q^{-\sigma}\prod_{p|q}\left(1+\frac{p^{\sigma}}{p}\right)\ll \begin{cases}
q^{-\sigma+\varepsilon} \quad &\sigma\in(0,1]\\
q^{-\sigma+\varepsilon}\prod_{p|q}p^{-1+\sigma}\ll q^{-1+\varepsilon} & \sigma\in(1,2).
\end{cases}
\end{align*}
Therefore
\begin{equation}\label{pree}
F_k(q,s)\ll q^{-\min(\sigma, 1)+\varepsilon}.
\end{equation}
It is obviously that $F_{k}(q,s)$ is analytic in $\cb$ for every $q$ which concerned. Hence one can use Cauchy estimate, say
\begin{equation}\label{cauchye}
\frac{{\rm d}^rF_{k}(q,s)}{{\rm d}s^r}{\bigg|}_{s=1}=\frac{r!}{2\pi i}\int_{|\xi-1|=\rho}\frac{F_{k}(q,\xi)}{(\xi-1)^{r+1}}{\rm d}\xi\ll \frac{r!}{\rho^r}\max_{\theta\in[0, 1)}\left|F_{k}(q,1+\rho e(\theta))\right|,
\end{equation}
where $\rho\in (0,1)$.
Hence combining with (\ref{pree}), we obtain that
\[\frac{{\rm d}^s F_k(q,1)}{{\rm d}s^r}\ll \frac{r!}{\rho^r}q^{-(1-\rho)+\varepsilon}\ll_{k} q^{-1+\varepsilon}.\]
Thus complete the proof of the lemma.
\end{proof}

 To apply the circle method, we need the following propositions.
\begin{proposition}\label{t23} Let $\alpha=a/q+\beta$ with $q\le X^{\frac{2}{k+1}}$ be an positive integer and $(a,q)=1$. Define
\[J_{k}(\alpha, X)=\sum_{m\le X}\tau_k(m)e(m\alpha).\]
Then
\[
J_{k}(\alpha, X)=\int_{1}^{X}e(u\beta){\rm Res}\left(\zeta(s)^kF_k(q,s)u^{s-1};s=1\right){\rm d}u+O_k\left(q(1+|\beta|X)X^{1-\frac{2}{k+1}+\varepsilon}\right),
\]
where $F_k(q,s)$ defined as Proposition \ref{t22}.
\end{proposition}
\begin{proof}
First, by Proposition \ref{t23} we have
\begin{align*}J_{k}(\alpha, X)&=\sum_{h\in\zb_q}e\left(\frac{ah}{q}\right)\sum_{\substack{m\le X\\ m\equiv h~(\mod q)}}\tau_k(m)e(m\beta)\\
&=\sum_{h\in\zb_q}e\left(\frac{ah}{q}\right)\int_{1}^Xe(u\beta){\rm d}\left(M_k(u;h,q)+O_k(u^{1-\frac{2}{k+1}+\varepsilon})\right)\\
&=\sum_{h\in\zb_q}e\left(\frac{ah}{q}\right)\int_{1}^Xe(u\beta)M'(u;h,q){\rm d}u+O_k\left(q(1+|\beta|X)X^{1-\frac{2}{k+1}+\varepsilon}\right).
\end{align*}
On the other hand,
\[\sum_{h\in\zb_q}e\left(\frac{ah}{q}\right)M'(u;h,q)=\sum_{h\in\zb_q}e\left(\frac{ah}{q}\right){\rm Res}\left(\zeta(s)^{k}u^{s-1}f_{k}(q,\delta,s); s=1\right),\]
where $\delta=(q,h)$ and using Proposition \ref{t22} we complete the proof of the lemma.
\end{proof}
The Riemann zeta function is meromorphic with a single pole of order one at $s = 1$. It can therefore be expanded as a Laurent series about $s = 1$, say
\[
\zeta(s)=\frac{1}{s-1}+\sum_{n=0}^{\infty}\frac{(-1)^{n}\gamma_n}{n!}(s-1)^n,
\]
where
\[\gamma_n=\lim_{M\rightarrow\infty}\left(\sum_{\ell=1}^{M}\frac{\log^{n}\ell}{\ell}-\frac{\log^{n+1}M}{n+1}\right),\;\; n\in\nb\]
are the Stieltjes constants. Therefore there exists constants $\alpha_{k,1}$, $\alpha_{k,2}$,...,$\alpha_{k,k}$ and a holomorphic function $h_k(s)$ on $\cb$ such that
\begin{equation}
\zeta(s)^{k}=\sum_{r=1}^{k}\frac{\alpha_{k,r}}{(s-1)^r}+h_k(s).
\end{equation}
Furthermore, we obtain that
\begin{equation}
\zeta(s)^{k}x^{s-1}=\sum_{r=1}^{k}\frac{1}{(s-1)^r}\sum_{r_1=0}^{k-r}\alpha_{k,r_1+r}\frac{\log^{r_1} x}{r_1!}+g_{k,x}(s),
\end{equation}
for any $x>0$, where $h_{k,x}(s)$ is a holomorphic function on $\cb$ about $s$. On the other hand, we also have a Taylor series for  $F_k(q,s)$ at $s=1$, say
\[F_k(q,s)=\sum_{\ell=0}^{\infty}\frac{F_k^{\langle \ell\rangle}(q,1)}{\ell !}(s-1)^{\ell}.\]
Therefore the residue of $\zeta(s)^{k}x^{s-1}F_k(q,s)$ at $s=1$ is
\begin{equation}
\sum_{\substack{r-\ell=1\\ \ell, r\in\nb,1\le r\le k}}\frac{F_k^{\langle \ell\rangle}(q,1)}{\ell !}\sum_{r_1=0}^{k-r}\alpha_{k,r_1+r}\frac{\log^{r_1} x}{r_1!}=\sum_{r=1}^{k}\frac{\log^{r-1}x}{(r-1)!}\sum_{t=0}^{k-r}F_k^{\langle t\rangle}(q,1)\frac{\alpha_{k,r+t}}{t!}.
\end{equation}
We Define
\[
\beta_{k,r}(q)=\frac{1}{r!}\sum_{t=0}^{k-r-1}\frac{\alpha_{k,r+1+t}}{t!}\left(\frac{{\rm d}^t F_k(q,s)}{{\rm d}s^t}\bigg|_{s=1}\right).
\]
Then by Proposition \ref{t22} we have $\beta_{k,r}(q)\ll q^{-1+\varepsilon}$ and the results of Proposition \ref{t23} rewritten as
\begin{lemma}\label{t24}Let $\alpha=a/q+\beta$ with $q\le X^{\frac{2}{k+1}}$ be an positive integer and $(a,q)=1$. Then, we have
\[
J_{k}(\alpha, X)=\sum_{r=0}^{k-1}\beta_{k,r}(q)\int_{1}^X(\log u)^re(u\beta){\rm d}u+O_k\left(q(1+|\beta|X)X^{1-\frac{2}{k+1}+\varepsilon}\right).
\]
where
\[\beta_{k,r}(q)=\frac{1}{r!}\sum_{t=0}^{k-r-1}\frac{\alpha_{k,r+1+t}}{t!}\left(\frac{{\rm d}^t F_k(q,s)}{{\rm d}s^t}\bigg|_{s=1}\right).\]
with
\[\alpha_{k,r}={\rm Res}\left((s-1)^{r-1}\zeta(s)^k; s=1\right)\;,\; F_k(q,s)=\sum_{\delta|q}\mu(q/\delta)f_k(q,\delta,s)\]
and where $f_k(q,\delta,s)$ defined by (\ref{fkqdef}).
\end{lemma}

The following lemmas will be used in the estimate of the major arcs of the circle method.
\begin{lemma}\label{t25} Let $\alpha=a/q+\beta$ with $q$ be an positive integer and $(a,q)=1$. Define
\[I_F(\alpha,X)=\sum_{{\bf x}\in B_{\ell}(X)}e\left(F({\bf x})\alpha\right).\]
Then
\[
I_F(\alpha,X)=q^{-\ell}S_F(q, a)\int_{[1,X]^{\ell}} e\left(F({\bf t})\beta\right){\rm d}{\bf t}+O_{F}\left(q(1+|\beta|X^2)X^{\ell-1}\right),
\]
where
\[S_F(q,a)=\sum_{{\bf h}\in (\zb_q)^{\ell}}e\left(\frac{a}{q}F({\bf h})\right).\]
\end{lemma}
\begin{proof}
Firstly,
\[I_F(\alpha,X)=\sum_{{\bf h}\in(\zb_q)^{\ell}}e\left(\frac{a}{q}F({\bf h})\right)\sum_{\substack{{\bf x}\in B_{\ell}(X)\\ {\bf x}\equiv {\bf h}~(\mod q)}}e\left(F({\bf x})\beta\right).\]
We shall prove
\begin{equation}\label{251}
\sum_{\substack{{\bf x}\in B_{\ell}(X)\\ {\bf x}\equiv {\bf h}~(\mod q)}}e\left(F({\bf x})\beta\right)-\frac{1}{q^{\ell}}\int_{[1,X]^{\ell}} e\left(F({\bf t})\beta\right){\rm d}{\bf t}\ll_F (1+|\beta|X^2)(X/q)^{\ell-1},
\end{equation}
which immediately yields the proof. For any $a, b\in\rb $ with $a\ll_F 1$ and $b\ll_F X$, let us consider the follows estimate
\[\sum_{\substack{m\le X\\ m\equiv h\bmod q}}e((am^2+bm)\beta)-q^{-1}\int_{1}^Xe((at^2+bt)\beta){\rm d}t\ll_F 1+|\beta|X^2,\]
which obtained by partial integration directly. So above applied successively for each variables $x_{i}\equiv h_i~(\mod q)$ yields \eqref{251}.
\end{proof}
\begin{lemma}\label{t26} Let $S_F(q,a)$ be defined as in Lemma \ref{t25} with $(q,a)=1$. Then for any $F({\bf x})$ be defined as above, we have
\[
S_F(q, a)\ll_F q^{{\ell}/{2}}.
\]
\end{lemma}
\begin{proof} Firstly,
\begin{align*}
\left|S_F(q, a)\right|^2&=S_F(q, a)\overline{S_F(q, a)}=\sum_{{\bf h}\in (\zb_q)^{\ell}}e\left(\frac{a}{q}F({\bf h})\right)\sum_{{\bf k}\in (\zb_q)^{\ell}}e\left(-\frac{a}{q}F({\bf k})\right)\\
&=\sum_{{\bf k}\in(\zb_q)^{\ell}}\sum_{{\bf h+k}\in(\zb_q)^{\ell}}e\left(\frac{a}{q}\left(F({\bf h+k})-F({\bf k})\right)\right).
\end{align*}
It is easily seen that
\[
F({\bf h+k})-F({\bf k})
=F({\bf h})-c+{\bf h}^t({Q_{m}+Q_m^t}){\bf k}.
\]
Hence we deduce that
\begin{align*}
\left|S_F(q, a)\right|^2&=\sum_{{\bf k}\in(\zb_q)^{\ell}}\sum_{{\bf h}\in(\zb_q)^{\ell}}e\left(\frac{a}{q}\left(F({\bf h})-c+{\bf h}^t({Q_{m}+Q_m^t}){\bf k}\right)\right)\\
&=q^{\ell}\sum_{{\bf h}\in(\zb_q)^{\ell},({Q_{m}+Q_m^t}){\bf h}\equiv {\bf 0}~(\mod q)}e\left(\frac{a}{q}\left(F({\bf h})-c\right)\right)\\
&\ll q^{\ell} \#\{{\bf h}\in(\zb_q)^{\ell}:({Q_{m}+Q_m^t}){\bf h}\equiv {\bf 0}\bmod q\}.
\end{align*}
Since $({Q_{m}+Q_m^t})$ is nonsingular, hence
\[
\left|S_F(q, a)\right|^2\ll q^{\ell}\#\{{\bf h}\in(\zb_q)^{\ell}:({Q_{m}+Q_m^t}){\bf h}\equiv {\bf 0}\bmod q\}\ll_{F} q^{\ell}.
\]
This completes the proof.
\end{proof}

To give a good estimate for $I_F(\alpha,X)$ in the minor arcs of the circle method, we need the following lemmas.
\begin{lemma}\label{t27}Let $A\in M_{\ell}(\zb)$ be a nonsingular matrix with column vectors ${\bf a}_1,...,{\bf a}_{\ell}$. Also let let $\alpha=a/q+\beta$ with $q$ be an positive integer, $(a,q)=1$ and $|\beta|\le q^{-2}$. Define
\[
H(X, A, \alpha)=\sum_{{\bf x}\in B(X)}\prod_{1\le v\le \ell}\min\left(X, \parallel {\bf a}_{v}^t{\bf x}\alpha\parallel^{-1}\right),
\]
where $B(X)=[1,X]^{\ell}\cap \zb^{\ell}$. Then
$$H(X, A, \alpha)\ll_{A} X^{2\ell} q^{-{\ell}}+X^{\ell}\log^{{\ell}}q+q^{\ell}\log^{\ell}q. $$
\end{lemma}
\begin{proof} Firstly we have
\[
\sum_{{\bf x}\in B(X)}\prod_{1\le v\le \ell}\min\left(X,\parallel {\bf a}_{v}^t{\bf x}\alpha\parallel^{-1}\right)\le
\sum_{{\bf x}\in B(X/q)}\sum_{{\bf h}\in(\zb_q)^{\ell}}\prod_{1\le v\le \ell}\min\left(X,\parallel {\bf a}_{v}^t(q{\bf x}+{\bf h})\alpha\parallel^{-1}\right).
\]
For the inner sum above
\begin{equation}\label{imeq}
U(X,A,{\bf x},q)=\sum_{{\bf h}\in(\zb_q)^{\ell}}\prod_{1\le v\le \ell}\min\left(X,\parallel {\bf a}_{v}^t(q{\bf x}+{\bf h})\alpha\parallel^{-1}\right)
\end{equation}
notes that $\alpha=a/q+\beta$, then
$${\bf a}_{v}^t(q{\bf x}+{\bf h})\alpha\equiv \frac{a{\bf a}_{v}^t{\bf h}}{q}+{\bf a}_{v}^t(q{\bf x}+{\bf h})\beta ~(\mod 1).$$
On the other hand, for each $v=1,2,...,\ell$, there exists some $t_v\in[0, 1-1/q)$ and uniquely ${\bf h}'=(h_{1}',...,h_{\ell}')^t\in\mathbb{Z}^{\ell}$ such that
\[
\left\{{\bf a}_{v}^t(q{\bf k}+{\bf h})\alpha\right\}={a{\bf a}_{v}^t{\bf h}}/{q}+{\bf a}_{v}^t(q{\bf x}+{\bf h})\beta-h_v'\in\left[t_v,t_v+1/q\right],
\]
namely
\begin{equation}\label{intvel}
a{\bf a}_{v}^t{\bf h}-qh_v'\in\left[qt_v-{\bf a}_{v}^t(q{\bf x}+{\bf h})q\beta, qt_v-{\bf a}_{v}^t(q{\bf x}+{\bf h})q\beta+1\right].
\end{equation}
In this case, it has uniformly ${\bf a}_{v}^t{\bf h}q\beta\ll_A 1$ for all ${\bf h}\in(\zb_q)^{\ell}$,
which implies the number of integers on above interval bounded by $O_A(1)$. For each $v$, let $N_v$ be an integer of above interval. If there exists an ${\bf h}\in(\zb_q)^{\ell}$ such that ${\bf a}_{v}^t{\bf h}-qh_v'=N_v$ for all $v=1,2,...,\ell$, namely
\begin{equation}\label{ieq}
A{\bf h}-q{\bf h}_v'={\bf N}_v,
\end{equation}
then $(a,q)=1$ and $A$ nonsingular implies
 \[\#\{{\bf h}\in(\zb_q)^{\ell}:aA{\bf h}\equiv {\bf N_v}\bmod q\}\ll_A 1.\]
Hence the number of ${\bf h}\in(\zb_q)^{\ell}$ satisfying (\ref{ieq}) be bounded by $O_A(1)$. Furthermore, for all $(t_1,t_2,...,t_{\ell})\in[0, 1-1/q)^{\ell}$, the number of ${\bf h}\in(\zb_q)^{\ell}$ satisfy the condition
\[
\left(\left\{{\bf a}_{v}^t(q{\bf x}+{\bf h})\alpha\right\},...,\left\{{\bf a}_{v}^t(q{\bf x}+{\bf h})\alpha\right\}\right)\in\left[t_1,t_1+1/q\right]\times...\times\left[t_v,t_v+1/q\right]
\]
bounded by $O_A(1)$. On the other hand $\parallel{\bf a}_{v}^t(q{\bf x}+{\bf h})\alpha\parallel\in\left[t_v,t_v+1/q\right]$ if and only if
\[
\left\{{\bf a}_{v}^t(q{\bf x}+{\bf h})\alpha\right\}\in\left[t_v,t_v+1/q\right]\;\mbox{or}\; 1-\left\{{\bf a}_{v}^t(q{\bf x}+{\bf h})\alpha\right\}\in\left[t_v,t_v+1/q\right].
\]
Hence for all $(t_1,t_2,...,t_{\ell})\in[0, 1-1/q)^{\ell}$, the number of ${\bf h}\in(\zb_q)^{\ell}$ satisfying $\parallel{\bf a}_{v}^t(q{\bf x}+{\bf h})\alpha\parallel\in\left[t_v,t_v+1/q\right]$ by bounded by $O_A(1)$.

For the convenience of discussion, $\forall s_1, s_2, . . . , s_{\ell}\in[0, q/2)\cap\mathbb{Z}$, we denote
\[K({\bf s})=\left[\frac{s_1}{q},\frac{s_1+1}{q}\right]\times\left[\frac{s_2}{q},\frac{s_2+1}{q}\right]\times\cdot\cdot\cdot\times
\left[\frac{s_{\ell}}{q},\frac{s_{\ell}+1}{q}\right]\]
and
\[A(q,{\bf x},{\bf h})=\left(\parallel{\bf a}_{1}^t(q{\bf x}+{\bf h})\alpha\parallel,...,\parallel{\bf a}_{\ell}^t(q{\bf x}+{\bf h})\alpha\parallel\right).\]
Then
\[\#\{{\bf h}\in(\zb_q)^{\ell}: A(q,{\bf x},{\bf h})\in K({\bf s})\}\ll_A 1.\]
Thus the sum (\ref{imeq}) can be rewritten as
\begin{align*}
U(X,A,{\bf x},q)&\ll \sum_{{\bf s}\in [0,q/2)^{\ell}\cap\mathbb{Z}^{\ell}}\sum_{A(q,{\bf x},{\bf h})\in K({\bf s})}\prod_{1\le v\le \ell}\min\left(X,\parallel {\bf a}_{v}^t(q{\bf x}+{\bf h})\alpha\parallel^{-1}\right)\\
&\ll\sum_{0\le m\le \ell}\sum_{\substack{{\bf s}\in [0,q/2)^{\ell}\cap\mathbb{Z}^{\ell}\\ {m~components~of~{\bf s}~not~0}}}\sum_{A(q,{\bf x},{\bf h})\in K({\bf s})}\prod_{1\le v\le \ell}\min\left(X,\parallel {\bf a}_{v}^t(q{\bf x}+{\bf h})\alpha\parallel^{-1}\right)\\
&\ll_A\sum_{0\le m\le\ell}\sum_{{\bf s}\in [1,q/2)^{m}\cap\mathbb{Z}^{m}}X^{\ell-m}\prod_{1\le v\le m}\min\left(X,\frac{q}{s_v}\right)\\
&\ll X^{\ell}\sum_{0\le m\le\ell}\sum_{{\bf s}\in [1,q)^{m}\cap\mathbb{Z}^{m}}\prod_{1\le v\le m}\frac{q}{Xs_v}\ll X^{\ell}\sum_{0\le m\le\ell}\left(\frac{q}{X}\right)^m\left(\sum_{1\le s\le q}\frac{1}{s}\right)^m\\
&\ll X^{\ell}\sum_{0\le m\le\ell}\left(\frac{q\log q}{X}\right)^m\le \ell( X^{\ell}+q^{\ell}\log^{\ell} q).
\end{align*}
Therefore we obtain that
\[
H(X,A,\alpha)\ll_{A} \sum_{{\bf x}\in B_{\ell}(X/q)}\left( X^{\ell}+q^{\ell}\log^{\ell} q\right)\ll \left(1+\frac{X}{q}\right)^{\ell}\left( X^{\ell}+q^{\ell}\log^{\ell} q\right).
\]
which completes the proof of the lemma.
\end{proof}
By this lemma, we have a nontrivial estimate for $I_F(\alpha, X)$ as follows.
\begin{lemma}\label{t28}Let $F({\bf x})$ defined by (\ref{fdef1}) and (\ref{fdef2}). Also let let $\alpha=a/q+\beta$ with $q$ be an positive integer, $(a,q)=1$ and $|\beta|\le q^{-2}$. Then, we have
\[
I_F(\alpha, X)\ll_{F} X^{\ell} q^{-\frac{\ell}{2}}+X^{\frac{\ell}{2}}\log^{\frac{\ell}{2}}q+q^{\frac{\ell}{2}}\log^{\frac{\ell}{2}}q.
\]
\end{lemma}
\begin{proof}First of all, we have that
\begin{align*}
\left|I_F(X,\alpha)\right|^2&=\ssum_{{\bf x,y}\in B_{\ell}(X)}e\left((F({\bf x})-F({\bf y}))\alpha\right)\\
&=\sum_{{\bf x}\in\mathbb{Z}^{\ell}\cap (-X,X)^{\ell}}\sum_{{\bf y},{\bf y+x}\in B_{\ell}(X)}e\left(\left(F({\bf x+y})-F({\bf y})\right)\alpha\right)\\
&=\sum_{{\bf x}\in\mathbb{Z}^{\ell}\cap (-X,X)^{\ell}}\sum_{{\bf y},{\bf x+y}\in B_{\ell}(X)}e\left(\left(F({\bf x})-c+{\bf y}^t(Q_m^t+Q_m){\bf x}\right)\alpha\right)\\
&\ll\sum_{{\bf h}\in\mathbb{Z}^{\ell}\cap (-X,X)^{\ell}}\left|\sum_{{\bf y},{\bf x+y}\in B_{\ell}(X)}e\left({\bf y}^t(Q_m^t+Q_m){\bf x}\alpha\right)\right|.
\end{align*}
Now we write the symmetric matrix ${Q_{mS}}=({\bf Q}_1, {\bf Q}_2, . . , {\bf Q}_{\ell})^t~({\bf Q}_j\in\mathbb{Z}^{\ell}, j=1,...,\ell)$. Then, using the fact that
\[
\sum_{a< n\le b}e(n\alpha)\ll\min\left(b-a+1,\parallel 2\alpha\parallel^{-1}\right)
\]
we obtain that
\begin{align*}
&\sum_{{\bf y},{\bf x+y}\in B_{\ell}(X)}e\left({\bf y}^t(Q_m^t+Q_m){\bf x}\alpha\right)\\
&\qquad=\sum_{\substack{ y_{r},x_{r}+y_r\in[1,X]\cap\mathbb{Z}\\r=1,...,\ell}}\left(\prod_{1\le v\le \ell}e\left(y_v{\bf Q}_{v}{\bf x}\alpha\right)\right)
=\prod_{1\le v\le \ell}\left(\sum_{y_{v},x_{v}+y_v\in[1,X]\cap\mathbb{Z}}e\left(y_v{\bf Q}_{v}{\bf x}\alpha\right)\right)\\
&\qquad\ll \prod_{1\le v\le \ell}\min\left(X-\left|h_{v}\right|,\parallel 2{\bf Q}_{v}{\bf x}\alpha\parallel^{-1}\right)\ll \prod_{1\le v\le \ell}\min\left({X},\parallel 2{\bf Q}_{v}{\bf x}\alpha\parallel^{-1}\right).
\end{align*}
Using the same method as in the proof of Lemma \ref{t28}, we can derive that
\begin{align*}
\sum_{{\bf x}\in\mathbb{Z}^{\ell}\cap (-X,X)^{\ell}}\prod_{1\le v\le \ell}\min\left({X},\parallel 2 {\bf Q}_{v}{\bf x}\alpha\parallel^{-1}\right)\ll_{F}X^{2\ell}q^{-\ell}+X^{\ell}\log^{\ell}q+q^{\ell}\log^{\ell}q.
\end{align*}
This completes the proof of the lemma.
\end{proof}
\section{Singular integral}
The well known results says that the gaussian integral
\[\int_{\rb^{\ell}}{\rm d}{\bf x}\exp\left(-{\bf x}^tA{\bf x}\right)\]
converges if $A$ is a symmetric complex matrix with the real part of $A$ is non-negative and no eigenvalue $\alpha_i$ of $A$ vanishes.
Hence we obtain that
\begin{align}\label{eq31}
&\int_{[0, 1]^{\ell}}e\left(\left({\bf t}^tQ_{m}{\bf t}+{{\bf b}^t}{\bf t}/{X}+{c}/{X^2}\right)\lambda\right)\mathrm{d}{\bf t}\nonumber\\
&\qquad=\left|\lambda\right|^{\frac{\ell}{2}}\int_{[0,\sqrt{\left|\lambda\right|}]^{\ell}}e\left(\left({\bf u}^tQ_{m}{\bf u}+\frac{{\bf b}^t\sqrt{\left|\lambda\right|}}{X}{\bf u}+\frac{c\left|\lambda\right|}{X^2}\right){\rm sign}(\lambda)\right)\mathrm{d}{\bf u}\ll_{F}\left|\lambda\right|^{-\frac{\ell}{2}},
\end{align}
where ${\rm sign}(\lambda)$ is general symbol function and $\lambda\neq 0$. We have the following
lemma.
\begin{lemma}\label{t31}We have
\[\int\limits_{\left|\beta\right|\le \frac{Q}{qX^2}}\int\limits_{[1,X]^{\ell}}\int\limits_{1}^{N_F(X)}{\rm d}\beta{\rm d}{\bf t}{\rm d}u e\left((F({\bf t})-u)\beta\right)(\log u)^r-\int\limits_{[1, X]^{\ell}}\mathrm{d}{\bf t}(\log(F({\bf t})))^r\ll_{k,F}X^{\ell+\varepsilon}\left(\frac{q}{Q}\right)^{\frac{\ell}{2}}.\]
\end{lemma}
\begin{proof}
Firstly, we have
\begin{align*}
&\int\limits_{\left|\beta\right|\le \frac{Q}{qX^2}}{\rm d}\beta\int_{[1,X]^{\ell}} e\left(F({\bf t})\beta\right){\rm d}{\bf t}\int_{1}^{N_F(X)}(\log u)^re(-u\beta){\rm d}u\\
&\qquad =X^{\ell}\int\limits_{\left|\beta\right|\le \frac{Q}{q}}{\rm d}\beta\int_{[1/X,1]^{\ell}} e\left(F({\bf t},X)\beta\right){\rm d}{\bf t}\int_{1/X^2}^{N_F(X)/X^2}(\log (X^2u))^re(-u\beta){\rm d}u,
\end{align*}
where $F({\bf t},X)={\bf t}^tQ_{m}{\bf t}+{{\bf b}^t}{\bf t}/{X}+{c}/{X^2}$. If $|\mu|>Q/q\ge 1$, using (\ref{eq31}) then
\begin{align*}
I_{r,F}(\mu,X)&=\int_{1/X^{2}}^{\frac{N_F(X)}{X^2}}(\log(uX^2))^re(-u\mu)\mathrm{d}u\int_{[1/X, 1]^{\ell}}e\left(F({\bf t},X)\mu\right)\mathrm{d}{\bf t}\\
&\ll_F {|\mu|^{-\frac{\ell+2}{2}}}\left|\int_{|\mu|/X^{2}}^{\frac{|\mu| N_F(X)}{X^2}}(|\log u|+\log |\mu|+2\log X)^re(-u)\mathrm{d}u\right|\\
&\ll_{k, F}|\mu|^{-\frac{\ell+2}{2}+\varepsilon}\log^{k} X\sum_{r=0}^k  \left|\int_{|\mu|/X^{2}}^{\frac{|\mu| N_F(X)}{X^2}}(\log u)^re(-u)\mathrm{d}u\right|\\
&\ll_{k, F}|\mu|^{-\frac{\ell+2}{2}+\varepsilon}\log^{k} X \log^k (|\mu|X)\ll_{k,F} |\mu|^{-\frac{\ell+2}{2}+\varepsilon}X^{\varepsilon}.
\end{align*}
The above result implies that
\begin{equation}\label{ierror}
\int_{|\mu|\le Q/q}I_{r,F}(\mu,X){\rm d}\mu=\int_{\mathbb{R}}I_{r,F}(\mu,X){\rm d}\mu+O_{k,F}\left(X^{\varepsilon}(Q/q)^{-\frac{\ell}{2}}\right).
\end{equation}
On the other hand,
\begin{align*}
I_{r,F}(X)&=\int_{\mathbb{R}}\mathrm{d}\mu\int_{1/X^{2}}^{\frac{N_F(X)}{X^2}}(\log(uX^2))^re(-u\mu)\mathrm{d}u\int_{[1/X, 1]^{\ell}}e\left(F({\bf t},X)\mu\right)\mathrm{d}{\bf t}\\
&=2\int_{\mathbb{R}_{\ge 0}}\mathrm{d}\mu\int_{1/X^{2}}^{\frac{N_F(X)}{X^2}}(\log(uX^2))^r\mathrm{d}u\int_{[1/X, 1]^{\ell}}\mathrm{d}{\bf t}\cos\left[2\pi(u-F({\bf t},X))\mu\right]\\
&=\frac{1}{\pi}\int_{[1/X, 1]^{\ell}}\mathrm{d}{\bf t}\int_{\mathbb{R}_{\ge 0}}\mathrm{d}\mu\int_{1/X^{2}}^{\frac{N_F(X)}{X^2}}(\log(uX^2))^r\mathrm{d}\left(\frac{\sin\left[2\pi(u-F({\bf t}, X))\mu\right]}{\mu}\right)\\
&=\frac{1}{\pi}\int_{[1/X, 1]^{\ell}}\mathrm{d}{\bf t}\int_{1/X^{2}}^{\frac{N_F(X)}{X^2}}(\log(uX^2))^r\mathrm{d}\left(\int_{\mathbb{R}_{\ge 0}}\mathrm{d}\mu\frac{\sin\left[2\pi(u-F({\bf t}, X))\mu\right]}{\mu}\right)\\
&=\frac{1}{\pi}\int_{[1/X, 1]^{\ell}}\mathrm{d}{\bf t}\int_{1/X^{2}}^{\frac{N_F(X)}{X^2}}(\log(uX^2))^r\mathrm{d}\left(\frac{\pi}{2}{\rm sign}(u-F({\bf t}, X))\right),
\end{align*}
where we have used the fact: $\int_{0}^{\infty}\frac{\sin(\alpha x)}{x}\mathrm{d}x=\frac{\pi}{2}{\rm sign}(\alpha)$. Note that
\begin{align*}
{\rm sign}(\alpha)=\begin{cases}\frac{\alpha}{\left|\alpha\right|} \qquad &\alpha\neq 0\\
~~0\qquad &\alpha=0,
\end{cases}
\end{align*}
then part integration yields
\begin{align*}
I_{r,F}(X)&=\frac{1}{2}\int_{[1/X, 1]^{\ell}}\mathrm{d}{\bf t}\int_{1/X^{2}}^{\frac{N_F(X)}{X^2}}(\log(uX^2))^r\mathrm{d}\left({\rm sign}(u-F({\bf t}, X))\right)\\
&=\frac{1}{2}\int_{[1/X, 1]^{\ell}}\mathrm{d}{\bf t}\int_{\substack{|u-F({\bf t},X)|\le \varepsilon\\ X^{-2}\le u\le \frac{N_F(X)}{X^2}}}(\log(uX^2))^r\mathrm{d}\left({\rm sign}(u-F({\bf t}, X))\right)\\
&=\frac{1}{2}\int_{[1/X, 1]^{\ell}}\mathrm{d}{\bf t}\left(2(\log(X^2F({\bf t},X)))^r+O_{k,F}(\varepsilon\log^rX)\right)\\
&=\int_{[1/X, 1]^{\ell}}\mathrm{d}{\bf t}(\log(F(X{\bf t})))^r=X^{-\ell}\int_{[1, X]^{\ell}}\mathrm{d}{\bf t}(\log(F({\bf t})))^r.
\end{align*}
Together with (\ref{ierror}) and above, we get the proof the lemma.
\end{proof}
\section{The proof of main theorem}
Here we refer the methods of Pleasants \cite{MR0209241} to deal with the minor arcs. Firstly, let $j\in\nb $ and define
\[
\mathfrak{M}(2^jQ)=\left\{\alpha\in[0,1]:\left|\alpha-a/q\right|\le{2^jQ}/({qX^2}), ~with~ q\le 2^jQ,~a\in \zb_q^*\right\}.
\]
It is obviously that,
$$\mathfrak{M}(2^jQ)\subseteq \mathfrak{M}(2^{j+1}Q)$$
and $\mathfrak{M}(2^jQ)=[0,1]$ when $j>\lfloor(\log (X/Q))/(\log 2)\rfloor:=N$ by well know Dirichlet's approximation theorem. If we define
\begin{align}\label{minor}
\mathcal{F}_{j}(Q)=\mathfrak{M}(2^{j+1}Q)\setminus\mathfrak{M}(2^jQ)
\end{align}
for $j=0,1,2,...,N$, then for all $i\neq j~(i,j=0,1,2,..,N)$ one has
\[
[0,1]=\mathfrak{M}(Q)\cup\left(\bigcup_{0\le j\le N}\mathcal{F}_{j}(Q)\right),\;\;\mathcal{F}_{i}(Q)\cap\mathcal{F}_{j}(Q)=\emptyset \;\mbox{and}\; \mathcal{F}_{i}(Q) \cap \mathfrak{M}(Q)=\emptyset.\]
We take $\mathfrak{M}(Q)$ as the major arcs, and the minor arcs is $\mathfrak{m}(Q)=[0,1]\setminus \mathfrak{M}(Q)$. As we all know, $\mathfrak{M}(2^jQ )$ is the union of all disjoint small intervals $\mathfrak{M}_j (q,a)$ with $1\le q\le 2^jQ$ and $(a,q)=1$, where $\mathfrak{M}_j (q,a)=[{a}/{q}-2^jQ ({qX^2})^{-1},{a}/{q}+2^jQ({qX^2})^{-1}]$. Thus we have
\[
\mathfrak{M}(2^jQ)=\bigcup_{1\le q\le 2^jQ}\bigcup_{a\in\zb_q^*}\mathfrak{M}_j(q, a)
\]
for all $j=0,1,...,N$ and $\mathfrak{m}(Q)=\bigcup_{j=0}^N\mathcal{F}_{j}(Q)$.  Therefore
\begin{align*}
T_{k,F}(X)&=\int_{0}^1I_F(\alpha,X)J_k(-\alpha,N_F(X))\mathrm{d}\alpha\\
&=\left\{\int_{\mathfrak{M}(Q)}+\int_{\mathfrak{m}(Q)}\right\}I_F(\alpha,X)J_k(-\alpha,N_F(X))\mathrm{d}\alpha:=T_{{\frak M},k,F}(X)+T_{{\frak m},k,F}(X).
\end{align*}
Applying the Cauchy-Schwarz inequality give an estimate for the minor arcs integral as follows
\begin{align}\label{eq33}
T_{{\frak m},k,F}(X)&=\int_{\mathfrak{m}(Q)}I_F(\alpha,X)J_k(-\alpha,N_F(X))\mathrm{d}\alpha=\sum_{j=0}^N\int_{\mathcal{F}_j(Q)}I_F(\alpha,X)J_k(-\alpha,N_F(X))\mathrm{d}\alpha\nonumber\\
&\le \sum_{j=0}^N \left(\left|\mathcal{F}_j(Q)\right|^{\frac{1}{2}}\sup_{\alpha\in \mathcal{F}_j(Q)}\left|I_F(\alpha,X)\right|\right)\left(\int_{0}^1\left|J_k(\alpha,N_F(X))\right|^2\mathrm{d}\alpha\right)^{\frac{1}{2}},
\end{align}
where $\left|\mathcal{F}_j(Q)\right|$ is the Lebesgue measure of the set $\mathcal{F}_j(Q)$. By (\ref{minor}) one has
\begin{equation}\label{eq34}
\left|\mathcal{F}_j(Q)\right|\le \left|\mathfrak{M}_{j+1}(Q)\right|\le \sum_{q\le 2^{j+1}Q}\varphi(q)\int_{\left|\lambda\right|\le\frac{2^{j+1}Q}{qX^2}}\mathrm{d}\lambda\le 4(2^jQX^{-1})^2.
\end{equation}
For $j\le N$, notes that $2^jQ\le X$ and  Lemma \ref{t28} one has
\begin{equation}\label{eq35}
\sup_{\alpha\in\mathcal{F}_j(Q)}\left|I_F(\alpha,X)\right|\ll_F X^{\ell}(2^jQ)^{-\frac{\ell}{2}}+X^{\frac{\ell}{2}}\log^{\frac{\ell}{2}}X.
\end{equation}
Hence by (\ref{eq34}), (\ref{eq35}) and $\ell\ge 3$ implies
\begin{align}\label{eq36}
\sum_{j=0}^N \left(\left|\mathcal{F}_j(Q)\right|^{\frac{1}{2}}\sup_{\alpha\in \mathcal{F}_j(Q)}\left|I_F(\alpha,X)\right|\right)&\ll_F\sum_{j=0}^N2^j\frac{Q}{X}\left(X^{\ell}(2^jQ)^{-\frac{\ell}{2}}+X^{\frac{\ell}{2}}\log^{\frac{\ell}{2}}X\right)\nonumber\\
& \ll X^{\ell-1}Q^{-\frac{\ell-2}{2}}+X^{\frac{\ell}{2}}\log^{\frac{\ell}{2}}X.
\end{align}
On the other hand
\[\int_{0}^1\left|J_k(\alpha,N_F(X))\right|^2\mathrm{d}\alpha=\sum_{m\le N_F(X)}\tau_k(m)^2\ll_{k,F} X^{2+\varepsilon},\]
hence together it with (\ref{eq33}) and (\ref{eq36}) we obtain that
\begin{equation}\label{E1}
T_{{\frak m},k,F}(X)\ll_{k,F} \left(Q^{-\frac{\ell-2}{2}}+X^{-\frac{\ell-2}{2}}\right)X^{\ell+\varepsilon}.
\end{equation}
For the major arc, by Lemma \ref{t24} and Lemma \ref{t25}, we have
\begin{align*}
T_{{\frak M},k,F}(X)&=\int_{\frak{M}(Q)}{\rm d}\alpha I_F(\alpha, X)J_{k}(-\alpha, N_F(X))\\
&=\sum_{q\le Q}\sum_{a\in\zb_q^*}\int\limits_{\left|\beta\right|\le \frac{Q}{qX^2}}{\rm d}\beta\left(\frac{S_F(q, a)}{q^{\ell}}\int_{[1,X]^{\ell}} e\left(F({\bf t})\beta\right){\rm d}{\bf t}+O_{F}\left(q(1+|\beta|X^2)X^{\ell-1}\right)\right)\\
&\times\left(\sum_{r=0}^{k-1}\beta_{k,r}(q)\int_{1}^{N_F(X)}(\log u)^re(-u\beta){\rm d}u+O_{k,F}\left(q(1+|\beta|X^2)X^{2-\frac{4}{k+1}+\varepsilon}\right)\right).
\end{align*}
Note that Lemma \ref{t26} and $\beta_{k,r}\ll_k q^{-1+\varepsilon}$, we obtain that
\begin{align*}
T_{{\frak M},k,F}(X)&=\sum_{r=0}^{k-1}\sum_{q\le Q}\beta_{k,r}(q)S_F(q)\int\limits_{\left|\beta\right|\le \frac{Q}{qX^2}}{\rm d}\beta\int_{[1,X]^{\ell}} e\left(F({\bf t})\beta\right){\rm d}{\bf t}\int_{1}^{N_F(X)}(\log u)^re(-u\beta){\rm d}u\nonumber\\
&+O_{k,F}\left(Q^2X^{\ell-\frac{4}{k+1}+\varepsilon}\right)+O_{k,F}\left(Q^2X^{\ell-1+\varepsilon}\right)+O_{k,F}\left(Q^4X^{\ell-1-\frac{4}{k+1}+\varepsilon}\right),
\end{align*}
where
\begin{equation}\label{sfq}
S_F(q)=\sum_{a\in\zb_q^*}q^{-\ell}S_F(q,a).
\end{equation}
On the other hand, by Lemma \ref{t31} we have
\begin{align*}
T_{{\frak M},k,F}(X)&=\sum_{r=0}^{k-1}\left(H_{k,r}(F)\int_{[1,X]^{\ell}}(\log F({\bf t}))^r{\rm d}{\bf t}+\sum_{q>Q}\beta_{k,r}(q)S_F(q)\int_{[1,X]^{\ell}}(\log F({\bf t}))^r{\rm d}{\bf t}\right)\\
&+O_{k,F}\left(X^{\ell+\varepsilon}Q^{-\frac{\ell-2}{2}}+Q^2X^{\ell-\frac{4}{k+1}+\varepsilon}+Q^2X^{\ell-1+\varepsilon}+Q^4X^{\ell-1-\frac{4}{k+1}+\varepsilon}\right)\\
&=\sum_{r=0}^{k-1}H_{k,r}(F)\int_{[1,X]^{\ell}}(\log F({\bf t}))^r{\rm d}{\bf t}\\
&+O_{k,F}\left(X^{\ell+\varepsilon}\left(Q^{-\frac{\ell-2}{2}}+Q^2X^{-\frac{4}{k+1}}+Q^2X^{-1}+Q^4X^{-\frac{k+5}{k+1}}\right)\right),
\end{align*}
where
\begin{equation}\label{ceff}
H_{k,r}(F)=\sum_{q=1}^{\infty}S_F(q)\beta_{k,r}(q).
\end{equation}
It is easily seen that when $Q=X^{\min(1,{4}/{(k+1)})/{(\ell+2)}}$, one has the optimal estimate
\[T_{k,F}(X)=\sum_{r=0}^{k-1}H_{k,r}(F)\int_{[1,X]^{\ell}}(\log F({\bf t}))^r{\rm d}{\bf t}+O_{k,F}\left(X^{\ell-\frac{\ell-2}{\ell+2}\min\left(1,\frac{4}{k+1}\right)+\varepsilon}\right).\]
We define
\[
L(s; k,F)=\sum_{q=1}^{\infty}S_F(q)F_k(q,s),
\]
then combine Lemma {\ref{t24}} and (\ref{ceff}) we obtain that
\[H_{k,r}(F)=\frac{1}{r!}\sum_{t=0}^{k-r-1}\frac{1}{t!}\left({\frac{{\rm d}^tL(s;k,F)}{{\rm d}s^t}}\bigg|_{s=1}\right) {\rm Res}\left((s-1)^{r+t}\zeta(s)^k; s=1\right).\]

We next try to give an explicit expression for $L(s; k,F)$.
\begin{lemma}\label{t41}The function $L(s; k,F)$ has the Euler product as follows
\[L(s; k,F)=\prod_{p}\left(1+\sum_{m=1}^{\infty}S_F(p^m)F_{k}(p^m,s)\right),\]
where
\[S_F(p^m)=p^{-(\ell-1)m}\varrho_F(p^m)-p^{-(\ell-1)(m-1)}\varrho_F(p^{m-1}),\]
\[\varrho_F(n)=\#\{{\bf h}\in (\zb_{n})^{\ell}: F({\bf h})\equiv 0\bmod n\}\]
for $n\in\nb_{\ge 1}$ and
\[F_k(p^m,s)=p^{-ms}\left(\sum_{v=1}^{k-1}(1-p^{-s})^{v-1}\tau_{v}(p^{m-1})+(1-p^{-s})^{k-1}\tau_k(p^{m-1})\frac{p^s-1}{p-1}\right).\]
\end{lemma}
\begin{proof}
It is easily seen that
\[S_F(q)=q^{-\ell}\sum_{a\in\zb_q^*}\sum_{{\bf h}\in(\zb_q)^{\ell}}e\left(\frac{a}{q}F({\bf h})\right)\]
is real and multiplicative.  When $q=p^m$ is a prime power with integer $m\ge 1$. It is easily seen that
\[
S_F(p^m)=p^{-(\ell-1)m}\varrho_F(p^m)-p^{-(\ell-1)(m-1)}\varrho_F(p^{m-1}).
\]
On the other hand, by Lemma \ref{t22} we shown that $F_k(q,s)$ is also multiplicative. Thus above implies the Euler product of $L(s; k,F)$. Applying Lemma \ref{t21} and Proposition \ref{t22}, we show that
\begin{align*}
F_{k}(p^m,s)&=f_k(p^m,p^m,s)-f_k(p^m,p^{m-1},s)\\
&=p^{-ms}\sum_{d_1...d_{k}=p^m}\prod_{i=1}^{k-1}\prod_{q|(\prod_{r=i+1}^{k}d_r),q~prime}\left(1-\frac{1}{q^s}\right)-\frac{p^s}{\varphi(p)p^{ms}}\left(1-\frac{1}{p^s}\right)^k\tau_k(p^{m-1}).
\end{align*}
For the first term above, denote by
\[
I_k=\sum_{d_1...d_{k}=p^m}\prod_{i=1}^{k-1}\prod_{q|(\prod_{r=i+1}^{k}d_r),q~prime}\left(1-\frac{1}{q^s}\right).\]
Then clearly for $m\ge 1$,
\[I_2=(m+1)(1-p^{-s}),\]
\begin{align*}
I_k&=\sum_{v=0}^{m}\sum_{d_1d_2...d_{k-1}=p^{m-v}}\prod_{i=1}^{k-1}\prod_{q|(p^{v}\prod_{r=i+1}^{k-1}d_r),q~prime}\left(1-\frac{1}{q^s}\right)\\
&=\sum_{d_1d_2...d_{k-1}=p^{m}}\prod_{i=1}^{k-2}\prod_{q|(\prod_{r=i+1}^{k-1}d_r),q~prime}\left(1-\frac{1}{p^s}\right)+\sum_{v=1}^{m}\sum_{d_1d_2...d_{k-1}=p^{m-v}}\left(1-\frac{1}{p^s}\right)^{k-1}\\
&=I_{k-1}+\left(1-p^{-s}\right)^{k-1}\sum_{v=1}^{m}\tau_{k-1}(p^{m-v})=I_{k-1}+\left(1-p^{-s}\right)^{k-1}\left(\tau_k(p^m)-\tau_{k-1}(p^m)\right)\\
&=(m+1)(1-p^{-s})+\sum_{v=3}^{k}\left(1-p^{-s}\right)^{v-1}\tau_{v}(p^{m-1})=\sum_{v=1}^{k}\left(1-p^{-s}\right)^{v-1}\tau_{v}(p^{m-1}).
\end{align*}
Hence
\[F_k(p^m,s)=p^{-ms}\left(\sum_{v=1}^{k-1}(1-p^{-s})^{v-1}\tau_{v}(p^{m-1})+(1-p^{-s})^{k-1}\tau_k(p^{m-1})\frac{p^s-1}{p-1}\right).\]
\end{proof}
Combining above estimates and calculations, we obtain the proof of the main theorem.

\end{document}